\documentclass{SCIYA2017enOL}
\online
\newcommand\bdf{\begin{definition}}
\newcommand\bpr{\begin{proposition}}
\newcommand\brk{\begin{remark}}
\newcommand\blm{\begin{lemma}}
\newcommand\bexe{\begin{exercise}}
\newcommand\bexa{\begin{example}}
\newcommand\beqn{\begin{eqnarray*}}
\newcommand\edf{\end{definition}}
\newcommand\epr{\end{proposition}}
\newcommand\erk{\end{remark}}
\newcommand\elm{\end{lemma}}
\newcommand\eexe{\end{exercise}}
\newcommand\eexa{\end{example}}
\newcommand\eeqn{\end{eqnarray*}}
\newcommand{\opt}{\rm Opt}
\newcommand{\geo}{\rm Geo}

\newcommand{\lmt}[2]{\mathop{\lim}_{{#1} \rightarrow {#2}} }

\newcommand{\lip}[1]{{\mathrm{lip}}({#1})}

\newcommand{\lmti}[2]{\mathop{\underline{\lim}}_{{#1} \rightarrow {#2}} }

\newcommand{\mm}{\mathfrak m}
\newcommand{\ms}{(X,\d,\mm)}

\newcommand{\cdkn}{{\rm CD}(0, \infty)}

\newcommand{\rcdkn}{{\rm RCD}(0, N)}
\newcommand{\rcd}{{\rm RCD}(0, \infty)}

\newcommand{\ent}[1]{{\rm Ent}_{#1}}
\newcommand{\N}{\mathbb{N}}

\newcommand{\qq}{\mathfrak{q}}
\newcommand{\R}{\mathbb{R}}

\newcommand{\diam}{\mathop{\rm diam}\nolimits} 

\newcommand{\supp}{\mathop{\rm supp}\nolimits}   %%\newcommand{\span}{\mathop{\rm span}\nolimits}   %
\newcommand{\Lip}{\mathop{\rm Lip}\nolimits}

\renewcommand{\d}{{\mathrm d}}

\newcommand{\D}{{\mathrm D}}
\newcommand{\restr}[1]{\lower3pt\hbox{$\llcorner_{#1}$}}

\newcommand{\nchi}{{\raise.3ex\hbox{$\chi$}}}

\begin{document}

\ensubject{fdsfd}

%%%%%%%%%%%%%%%%%%%%%%%%%%%%%%%%%%%%%%%%%%%%%%%%%%%%%%%
%%% Authors do not modify the information below
%%%
%%%
\ArticleType{ARTICLES}%ÀžÄ¿
%\SpecialTopic{Progress of Projects Supported by NSFC}
%\SubTitle{Dedicated to Professor Aaaaa Bbbbb on the Occasion of his {\rm 80}th Birthday}
\Year{2026}
\Month{January}%
\Vol{69}
\No{1}
\BeginPage{1} %
\DOI{10.1007/s11425-023-xxxx-x}
\ReceiveDate{January 1, 2026}
\AcceptDate{January 1, 2026}
\OnlineDate{January 1, 2026}
%%%%%%%%%%%%%%%%%%%%%%%%%%%%%%%%%%%%%%%%%%%%%%%%%%%%%%%

%%% title:
%%%   \title{title}{title for citation}
\title{Sharp and rigid isoperimetric inequalities in metric measure spaces with non-negative Ricci curvature}
{Sharp and rigid isoperimetric inequalities in metric measure spaces with non-negative Ricci curvature}%%

%%% Corresponding author:
%%%   \author[number]{Full name}{{email@xxx.com}}
%%% General author:
%%%   \author[number]{Full name}{}
\author[1, $\ast$]{Bang-Xian Han}{{hanbx@sdu.edu.cn}}

%%% Author information for page head.
\AuthorMark{Bang-Xian Han}

%%% Authors for citation.
\AuthorCitation{Bang-Xian Han}

%%% Address.
%%%   \address[number]{Address, City {\rm Postcode}, Country}
\address[1]{School of   Mathematics, Shandong University,\\
Jinan, {\rm 250100}, China}

%%% Abstract.
\abstract{Using optimal transport, we prove a sharp, dimension-free isoperimetric inequality involving volume entropy for metric measure spaces with non-negative Ricci curvature in the sense of Lott--Sturm--Villani. 
Furthermore, we prove rigidity for ${\rm RCD}(0, \infty)$ spaces via Bakry--\'Emery's $\Gamma_2$-calculus. These results are new even for Euclidean spaces equipped with log-concave densities, and are of interest in both probability theory and geometry.}

%%% Keywords.


\maketitle

\section{Introduction}

The aim  of this  paper is to present a \emph{sharp dimension-free  isoperimetric inequality}, in metric measure spaces with non-negative Ricci curvature in the sense of Lott\textendash Sturm\textendash Villani, and study its rigidity.

Let $\ms$ be a metric measure space, where  $(X, \d)$ is a complete and separable metric space and $\mm$ is  a locally finite, non-negative Radon measure with full support.  The \emph{Minkowski content} of a  Borel
set $\Omega \subset X$ with $\mm(\Omega)<+\infty$ is defined by
\[
\mm^+(\Omega):=\mathop{\liminf}_{\epsilon \to 0^+} \frac{\mm(\Omega^\epsilon)-\mm(\Omega)}{\epsilon}
\]
where $\Omega^\epsilon \subset X$ is the $\epsilon$-neighbourhood of $\Omega$ defined as $\Omega^\epsilon:=\{x: \d(x, \Omega)<\epsilon\}$.  An isoperimetric inequality   relates  the size of the boundary of a set to its measure. Let  $\mathcal M$ be a family of metric measure spaces. We say that a function $I_{\mathcal M}(\cdot): [0, +\infty) \to  [0, +\infty)$ is \emph{the isoperimetric profile} associated with $\mathcal M$,  if
\[
\mm^+(\Omega) \geq I_{\mathcal M}(v)
\]
for all $\ms \in \mathcal M$ and any measurable set $\Omega \subset X$ with $\mm(\Omega)=v$.

Recently, isoperimetric inequalities in  non-compact metric measure spaces with  non-negative  synthetic Ricci curvature,  have been studied in various settings, for example by Agostiniani\textendash Fogagnolo\textendash Mazzieri\cite{AFM-S}, Brendle\cite{BrendleSobolev}, Balogh and  Krist\'aly\cite{Balogh:2022aa},  Antonelli\textendash Pasqualetto\textendash Pozzetta\textendash Semola\cite{antonelli2022asymptotic, antonelli2022sharp},   Cavalletti and Manini\cite{cavalletti2021isoperimetric, cavalletti2022rigidity}.  As discovered by E. Milman\cite{Milman-JEMS}, the isoperimetric profile  for this family of spaces is trivial if there is no restriction on the diameter of the sets.
In the  above-mentioned papers, a key component  in the isoperimetric profile is  a  parameter called  \emph{the asymptotic volume ratio}.

However,  the asymptotic volume ratio  depends on the dimension parameter, so those  isoperimetric inequalities   are all \emph{dimension-dependent}. So it is natural to ask for a \emph{dimension-free} isoperimetric inequality in metric measure spaces with  non-negative Ricci curvature,   in the sense of Lott\textendash Sturm\textendash Villani\cite{Lott-Villani09, S-O1}.  Examples satisfying this condition include weighted Riemannian manifolds with non-negative Bakry\textendash \'Emery curvature, Banach spaces,  measured Gromov--Hausdorff limits of Riemannian manifolds with non-negative Ricci curvature,    Alexandrov spaces with  non-negative curvature,  and Finsler manifolds with non-negative weighted Ricci curvature. See  Ambrosio's  ICM Proceeding\cite{AmbrosioICM} and Villani's book\cite{V-O}  for  an overview of this field.

\begin{definition}[Lott\textendash Sturm\textendash Villani\cite{Lott-Villani09, S-O1}]\label{def:cd}
We say that a metric measure space $\ms$  has non-negative Ricci curvature, or satisfies ${\rm CD}(0, \infty)$ condition,  if  the relative entropy    $\ent{\mm}$ defined as
\begin{equation*}
{\rm Ent}_\mm(\mu):=
\left \{\begin{array}{ll}
\int \ln \rho\,\d \mu &\text{if}~ \mu=\rho\,\mm\\
+\infty &\text{otherwise}
\end{array}\right.
\end{equation*}  is displacement convex. This is to say,  for any two probability measures $\mu_0, \mu_1$ in the $L^2$-Wasserstein space $(\mathcal{P}_2(X), W_2)$, there  is  a geodesic ${(\mu_t)}_{t\in [0,1]}$ satisfying
 \begin{equation*}
{\rm Ent}_\mm(\mu_t) \leq t{\rm Ent}_\mm(\mu_1)+(1-t){\rm Ent}_\mm(\mu_0)~~\forall t\in [0,1].
\end{equation*}
\end{definition}

\begin{remark}\label{spp}
Let $(X,\d)$ be a  geodesic space and fix  $\mu_{0}, \mu_{1} \in \mathcal{P}_2(X)$. A curve $(\mu_t)_{t \in [0,1]} \subset (\mathcal{P}_2(X), W_2)$
is a  geodesic, i.e.,
\begin{equation*}
{W}_2(\mu_s,\mu_t)=|s-t| W_2(\mu_0,\mu_1), \quad \forall s,t \in [0,1],
\end{equation*}
if and only if there exists  $\Pi \in  \mathcal{P}(\geo(X)) \subseteq \mathcal{P}(C([0,1],X)) $, called the optimal dynamical plan, such that 
\[
\mu_t = ({{\rm e}_t})_\sharp \Pi \; \; \forall t \in [0,1]  \quad \text{and}\quad  ({{\rm e}_0}, {{\rm e}_1})_\sharp \Pi \in \opt(\mu_0, \mu_1),
\]
where  $\opt(\mu_0, \mu_1)$ denotes the set of optimal transport plans from $\mu_{0}$ to $\mu_{1}$.
\end{remark}

In order to evaluate the growth of the volume without the dimension parameter, we  will use  \emph{volume entropy}. This is an important concept in both Riemannian  geometry (cf.\cite{BessonEntropy}) and  dynamical systems (cf.\cite{ManningEntropy}). For example, it is related to Gromov's simplicial volume, the bottom of the spectrum of  the Laplacian, the Cheeger isoperimetric constant, the growth of fundamental groups, and topological entropy of geodesic flows.
\begin{definition}[Volume entropy]\label{def:ve}
We say that a metric measure space $\ms$ has {volume entropy} $h_{\ms}$,  if there is $x_0\in X$ so that the following limit exists
\[
h_{\ms}:=\lmt{r}{+\infty} \frac {\ln \mm\big (B_{r}(x_0)\big)} {r}\in [0,\infty].
\]

\end{definition}

\begin{remark}
While the existence of volume entropy is well-known for the universal cover of a compact Riemannian manifold
\cite{BessonEntropy,ManningEntropy}, to the best of our knowledge, this is the first result 
establishing the existence of volume entropy for general non-compact 
$\mathrm{CD}(0,\infty)$ metric measure spaces without assuming a finite 
upper dimension bound. See Proposition \ref{prop:entropy} below 
for the proof.
\end{remark}

\medskip

The first main result of this paper is the following sharp  isoperimetric inequality involving volume entropy. 

\begin{theorem}[Sharp isoperimetric inequality, Theorem \ref{th1} and Theorem \ref{th2}]\label{th-1}
Let $\ms$ be a metric measure space satisfying the $\cdkn$ condition.  Then for any $\Omega \subset X$ with $\mm(\Omega)<\infty$, we have 
\begin{equation}\label{intro:eq1}
\mm^+(\Omega) \geq h_{\ms} \mm(\Omega).
\end{equation}
In other words, the {Cheeger constant}  $\mu_{\ms}:=\inf_{\Omega} \frac{ \mm^+(\Omega) }{\mm(\Omega)}$ is 
no less than the volume entropy $h_{\ms}$.
Moreover,  the constant $h_{\ms}$ in \eqref{intro:eq1} cannot be improved. 
\end{theorem}

\medskip

In \cite[Theorem 1]{Brooks1981}, R. Brooks proved that the bottom of the essential spectrum $\lambda_0^{\rm ess}$ is bounded from above by $\frac 14 h_{\ms}^2$ if $\mm(X)=+\infty$. 
Combining this with Cheeger's inequality \cite{Cheeger1970} we get the following inequality (cf. \cite[Corollary 2]{Brooks1981})
\[
\frac 14 h_{\ms}^2 \geq \lambda_0^{\rm ess} \geq \frac 14 \mu_{\ms}^2.
\]
Then  we obtain the following corollary.

\begin{corollary}
Let $\ms$ be a   $\cdkn$ metric measure space with  infinite volume. The following equality holds:
\[
\frac 14 h_{\ms}^2 = \lambda_0^{\rm ess} = \frac 14 \mu_{\ms}^2.
\]
\end{corollary}

\medskip

It has  been observed by De Ponti\textendash Mondino\textendash Semola  \cite{DPMS2021} (see also \cite{DPM2021}) that  the equality in Cheeger's isoperimetric inequality can never be attained in the family of 
 spaces with finite diameter or positive Ricci curvature (Theorem  \ref{th-1} provides another interpretation of this fact).  In the next theorem we show that, in  metric measure spaces satisfying the Riemannian curvature-dimension condition $\rcd$,  the  isoperimetric inequality is rigid.   Here  ``Riemannian" means that $(X,\d,\mm)$ is  infinitesimally Hilbertian (cf. \cite{AGS-M, G-O}).  In this case,  we can use non-smooth Bakry--\'Emery theory and $\Gamma_2$-calculus to prove the splitting theorem.

\begin{theorem}[Rigidity theorem, Theorem \ref{th:rigid}]\label{th-2}
Let $\ms$ be an $\rcd$  metric measure space having positive volume entropy $h_{\ms}$.

 If there is a measurable set $\Omega \subset X$ with $\mm(\Omega)<\infty$ such that the equality in the isoperimetric inequality \eqref{intro:eq1} is attained 
\[
{\mm^+(\Omega) } =h_{\ms}  {\mm(\Omega)}
\]
or the Cheeger constant is achieved
\[
\mu_{\ms}=\frac{ \mm^+(\Omega) }{\mm(\Omega)},
\]
then 
\[\ms \cong \Big (\R, | \cdot |, {\mathrm e}^{h_{\ms} t} \d t \Big) \times (Y, \d_Y, \mm_Y)\] for some ${\rm RCD}(0, \infty)$  metric measure  space $(Y, \d_Y, \mm_Y) $ with $\mm_Y(Y)<+\infty$.
In  a suitable choice of coordinates, $\Omega$ can be identified as
\[
\Omega={(-\infty, c] \times Y} \subset  \R \times Y
\]
with $c\in \R$ satisfying  $\mm_Y(Y) {\mathrm e}^{h_{\ms} c} =h_{\ms}\mm(\Omega)$.
\end{theorem}

\bigskip

The rest of this paper is organized as follows.  In Section \ref{sect:main} we   prove the sharp isoperimetric inequality and its  corollaries.   In Section \ref{sect:rigidity} we  study  its  rigidity.

\bigskip

\noindent \textbf{Acknowledgement}:  The author thanks  Nicol\`o  De Ponti for bringing  his attention to  the paper  of Robert Brooks,  and thanks Giocchiano Antonelli, Sara Farinelli and Marco Pozzetta for their interest and suggestions on this paper.  The author  also acknowledges the valuable suggestions from the peer reviewers. This  work is supported in part by  the National key R \& D programs of China (2021YFA1000900, 2021YFA1002200),  National Natural Science Foundation of China  (12201596),  Shandong Provincial Natural Science Foundation (ZR2025QB05) and Taishan Scholars Program of Shandong Province (tsqn202408059).

\section{Sharp Isoperimetric Inequalities}\label{sect:main}
In this section we will prove a sharp isoperimetric inequality in $\cdkn$ metric measure spaces. 

\begin{theorem}[Sharp isoperimetric inequality]\label{th1}
Let $\ms$ be a  $\cdkn$ metric measure space having  volume entropy $h_{\ms}$.  Then for any measurable set $\Omega \subset X$ with $\mm(\Omega)<\infty$, the following isoperimetric inequality holds:
\begin{equation}\label{th1:eq1}
\mm^+(\Omega) \geq h_{\ms} \mm(\Omega).
\end{equation}
\end{theorem}
\begin{proof}
{\bf Step 1.}
Assume $\Omega$ to be  bounded.

Given $x_0 \in \Omega$, let $R>0$ be such that $\Omega \subset B_R(x_0):=\{x: \d(x, x_0)<R\}$.
Define $\mu_0=\frac 1{\mm(\Omega)} \mm \restr{\Omega}$ and $\mu_1=\frac 1{\mm(B_{R}(x_0))} \mm \restr{B_{R}(x_0)}$.  According to  Definition \ref{def:cd},  there exists an  $L^2$-Wasserstein geodesic  $(\mu_t)_{t\in [0, 1]}$ connecting $\mu_0, \mu_1$ such that
 \begin{equation}\label{th1:eq2}
{\rm Ent}_\mm(\mu_t) \leq t{\rm Ent}_\mm(\mu_1)+(1-t){\rm Ent}_\mm(\mu_0).
\end{equation}
Denote the set of $t$-intermediate points by \[Z_t:=\left \{z: \exists~ x\in \Omega, y\in B_{R}(x_0), \text{such that}~ \frac {\d(z, x)}t=\frac{\d(z, y)}{1-t}=\d(x, y) \right\}.\]
It can be seen from the super-position theorem (cf. \cite[Theorem 2.10]{AG-U} and Remark \ref{spp})  that $ \mu_t $ is concentrated on $Z_t$. Then by \eqref{th1:eq2}, Jensen's inequality  and monotonicity of the function $t \to \ln (t)$  we have
 \begin{equation}\label{th1:eq3}
-\ln \big (\mm(Z_t)\big) \leq  -t \ln \big (\mm(B_{R}(x_0))\big ) - (1-t) \ln \big(\mm(\Omega) \big).
\end{equation}

Let $\epsilon:=t\big (\diam (\Omega)+R\big)$.  For any $z\in Z_t$,  there are $x\in \Omega$ and $y\in B_R(x_0)$ so that $\d(z, x)=t\d(x,y)$.  By the triangle inequality,  $$\d(x, y)\leq \d(x, x_0)+\d(y, x_0)<\diam ( \Omega)+R.$$So  $\d(z,x)<\epsilon$,  $z\in \Omega^\epsilon$ and  $Z_t\subset \Omega^\epsilon$.   

If $\mm^+(\Omega)=+\infty$, there is nothing to prove. Otherwise,
$\lmt{\epsilon}0\mm(\Omega^\epsilon)=\mm(\Omega)$. So we have
\begin{eqnarray*}
\frac{\mm^+(\Omega) }{\mm(\Omega)}&=& \mathop{\liminf}_{\epsilon \to 0}  \frac 1{\mm(\Omega) } \frac{\mm(\Omega^\epsilon)-\mm(\Omega)}{ \epsilon}\\
 &\stackrel{\text {L'H\^opital}}{=}&   \mathop{\liminf}_{\epsilon \to 0} \frac{\ln \big(\mm(\Omega^\epsilon)\big)- \ln \big(\mm(\Omega)\big)}  {\mm(\Omega^\epsilon)-\mm(\Omega)}  \frac {\mm(\Omega^\epsilon)-\mm(\Omega)}{ \epsilon}\\
&\geq &    \mathop{\liminf}_{t \to 0} \frac{\ln\big(\mm(Z_t)\big)-\ln \big(\mm(\Omega) \big)}{t(\diam (\Omega)+R)}\\
 &\stackrel{\eqref{th1:eq3}}{\geq}&  \mathop{\liminf}_{t \to 0} \frac{t \ln \big (\mm(B_{R}(x_0))\big ) +(1-t) \ln \big(\mm(\Omega) \big)-\ln\big(\mm(\Omega)\big)}{t(\diam (\Omega)+R)}\\
&=&   \frac{ \ln \big (\mm(B_{R}(x_0))\big )-\ln\big(  \mm(\Omega)\big)}{\diam (\Omega)+R}.
\end{eqnarray*}
By Proposition \ref{prop:entropy},  the volume entropy $h_{\ms}$ exists.  Letting $R \to \infty$,  we get
\begin{equation}\label{eq:mmh}
\frac{\mm^+(\Omega) }{\mm(\Omega)} \geq h_{\ms}
\end{equation}
which proves the claim.

\bigskip

{\bf Step 2.}
For  $\Omega \subset X$ with $\mm(\Omega)<\infty$, we adopt an  argument used by Cavalletti\textendash Manini  \cite[Theorem 3.2]{cavalletti2021isoperimetric},   based on a  relaxation principle investigated in \cite[Theorem 3.6]{ADMG-P}.  For any  $\Omega \subset X$ with $\mm(\Omega)<\infty$, we have
\begin{eqnarray*}
{\rm Per}(\Omega)
&=& \inf \left\{ \liminf_{n\to \infty} \int \lip{f_n}\,\d \mm: f_n\in \Lip(X,\d),~\lim_{n\to \infty}\int|f_n-\chi_\Omega|\,\d \mm=0 \right\}\\
&=&\inf \left\{ \liminf_{n\to \infty} \mm^+(\Omega_n): \mm(\Omega \Delta \Omega_n) \to 0 \right\}
\end{eqnarray*}
where we   require $\Omega_n$ to be bounded.
Applying \eqref{eq:mmh} with $\Omega_n$ and letting $n\to \infty$ we get
\begin{equation}\label{eq:per}
{\rm Per}(\Omega) \geq h_{\ms} \mm(\Omega).
\end{equation}
By \cite[Theorem 3.6]{ADMG-P},  ${\rm Per}(\Omega) \leq \mm^+(\Omega)$, we complete the proof.
\end{proof}
\bigskip

In the same spirit we can prove the existence of the volume entropy on ${\rm CD}(0, \infty)$ spaces. Note that on Riemannian manifolds having lower Ricci curvature bounds, this result can be proved using  the Bishop--Gromov volume comparison inequality (see \cite{BessonEntropy} and   \cite{ManningEntropy}). 
\begin{proposition}\label{prop:entropy}
Let $\ms$ be a   ${\rm CD}(0, \infty)$ metric measure space. Then  the volume entropy $h_{\ms} \in [0, +\infty]$ exists in the sense of Definition \ref{def:ve}.
\end{proposition}

\begin{proof}
 If $\mm(X)<+\infty$,  we have $h_{\ms}=0$, otherwise there is $\epsilon>0$ such that $\mm(B_\epsilon (x_0))>1$.  
Applying \eqref{th1:eq3} with  $\Omega=B_{r+\delta}(x_0)$  and $R=\epsilon$ for some $r>\epsilon$ and $\delta,>0$,   we get 
 \begin{equation}\label{prop1:eq1}
\ln \big (\mm(Z_t)\big) \geq  t \ln \big (\mm(B_\epsilon (x_0))\big ) +(1-t) \ln \big(\mm(B_{r+\delta}(x_0)) \big)~~\forall t\in [0, 1].
\end{equation}
For any $z\in Z_t$,  by the triangle inequality  
\[
\d(z, x_0) < (1-t)[(r+\delta)+\epsilon]+\epsilon.
\]
So for  $t=\frac {\delta+\epsilon}{r+\delta} $,   we have $Z_t \subset B_{r+\epsilon}(x_0)$. Thus  \eqref{prop1:eq1} implies
\begin{equation}\label{eq:lemma:ve}
\ln \big (\mm(B_{r+\epsilon} (x_0))\big )  \geq  \frac {\delta+\epsilon}{r+\delta}  \ln \big (\mm(B_\epsilon (x_0))\big ) +\frac {r-\epsilon}{r+\delta} \ln \big(\mm(B_{r+\delta}(x_0)) \big).
\end{equation}
Dividing   both sides of \eqref{eq:lemma:ve} by $r$, we get
\[
 \frac{\ln \big (\mm(B_{r+\epsilon} (x_0))\big )}{r} \geq \left (1-\frac \epsilon r\right) \frac{\ln \big (\mm(B_{r+\delta} (x_0))\big )}{r+\delta}~~~~\forall r, \delta>0.
\]
Then
\[
\liminf_{r\to +\infty}  \frac{\ln \big (\mm(B_{r} (x_0))\big )}{r}  \geq \lim_{r\to +\infty}   \left (1-\frac \epsilon r\right) \limsup_{\delta \to +\infty}  \frac{\ln \big (\mm(B_{\delta} (x_0))\big )}{\delta},
\]
which is the thesis.
\end{proof}
Next we will show that the inequality \eqref{th1:eq1} is sharp. This can be proved by  combining    \cite[Corollary 2]{Brooks1981} where  Brooks showed that the  Cheeger constant is no larger than the volume entropy,  and our Theorem \ref{th1}.  We will give a  different  proof  which  has its own interest.

\begin{theorem}[Sharpness]\label{th2}
The inequality \eqref{th1:eq1} in Theorem \ref{th1} is sharp. This means that,  for  any  $\cdkn$ space $\ms$  and  $C> h_{\ms}$, the inequality ${\mm^+(\Omega) } \geq C {\mm(\Omega)}$ does not always hold.
\end{theorem}
\begin{proof}
We will prove the theorem by contradiction. Assume there is a constant $C> h_{\ms}$, such that 
 \begin{equation}\label{th2:eq1}
{\mm^+(\Omega) } \geq C {\mm(\Omega)}>0
\end{equation}
  for any bounded set $\Omega \subset X$.

By \eqref{eq:lemma:ve} we have
\begin{eqnarray*}
&&\frac {r-\epsilon}{r+\delta}\left ( \frac{\ln \big(\mm(B_{r+\delta}(x_0)) \big)-\ln \big(\mm(B_{r}(x_0)) \big)}{\delta} \right )\\
&\leq &  \frac {\delta+\epsilon}{\delta(r+\delta)} \Big( \ln \big (\mm(B_{r} (x_0))\big ) -  \ln \big (\mm(B_\epsilon (x_0))\big ) \Big).
\end{eqnarray*}
Applying \eqref{th2:eq1} with geodesic balls,  we get 
\[
{\ln \big(\mm(B_{r+\delta}(x_0)) \big)-\ln \big(\mm(B_{r}(x_0)) \big)}  \geq  \delta C,
\]
so
\[
\frac {r-\epsilon}{r+\delta} C\leq   \frac {\delta+\epsilon}{\delta(r+\delta)} \Big( \ln \big (\mm(B_{r} (x_0))\big ) -  \ln \big (\mm(B_\epsilon (x_0))\big ) \Big).
\]
Letting $r \to \infty$, we get
\[
C \leq   \frac {\delta+\epsilon}{\delta} h_{\ms}.
\]
Letting $\epsilon \to 0$ we get the contradiction.
\end{proof}

\section{Rigidity}\label{sect:rigidity}
In this section we will study the rigidity of  the  isoperimetric inequality  \eqref{th1:eq1}.  
We first deal with the  rigidity for 1-dimensional spaces. The idea behind its proof is essential,  which will be used directly or indirectly later.   Then we will study the equality case of the isoperimetric inequality in the general  $\rcd$ setting.

\subsection{Rigidity for log-concave densities}\label{subsec:1dim}
By a well-known  result of Bobkov \cite{Bobkov96},  for any log-concave density on $\R$, the infimum   in the corresponding isoperimetric problem  is attained by a  half line. Among all  log-concave densities, the log-linear densities ${\mathrm e}^{ht}$ play  a special role.

\begin{proposition}[Rigidity for log-concave densities]\label{lemma:rigidity}
Let  $\ms=\Big (\R, | \cdot |, {\rm e}^{f} \mathcal L^1 \Big) $ be a 1-dimensional metric measure space,  where ${f} $ is concave and
\[
\lmt{t}{+\infty} f'(t)= h>0.
\]
Then the volume entropy $h_{\ms}=h$. If there is $\Omega \subset \R$ such that 
\[
\mu_{\ms}=\frac{ \mm^+(\Omega) }{\mm(\Omega)}=h,
\]
then $f'= h$ and $\Omega =(-\infty, b)$  for some $b\in \R$.
\end{proposition}

\begin{proof}
Since $f$ is concave,  $f'$ is well-defined almost everywhere, and the limits $\lmt{t}{-\infty} f'(t)$, $\lmt{t}{+\infty} f'(t)$ exist. Assume $\lmt{t}{+\infty} f'(t)=h>0$.   We can see that the volume entropy of $\ms$ is the same as the volume entropy of $\Big (\R, | \cdot |, {\mathrm e}^{ht}\d t\Big)$,  which is exactly $h$. By Theorem \ref{th1} and Theorem \ref{th2} we know  the  Cheeger constant $\mu_{\ms}$ is   $h$.  

Assume there is $\Omega \subset \R$ attaining $\mu_{\ms}$,  and by Bobkov's result \cite{Bobkov96}, $\Omega$ must be  $(-\infty, \rm C)$ for some $C\in \R$. Assume by contradiction that $f'({\rm C})>h$. We replace $f$ by 
\begin{equation*}
\tilde f(t):=
\left \{\begin{array}{ll}
f(t) &t\leq {\rm C},\\
f({\rm C})+f'({\rm C})(t-{\rm C}) &t>{\rm C}.
\end{array}\right.
\end{equation*}
Similarly, we can see that the volume entropy of $\Big (\R, | \cdot |, {\mathrm e}^{\tilde f(t)}\d t\Big)$ is $f'(\rm C)$ and its corresponding Cheeger constant $\mu_{{\mathrm e}^{\tilde f(t)} }\geq f'({\rm C}) >h$. However,  by the definition of the Cheeger constant, 
\begin{equation}\label{eq1d:intro}
\mu_{{\mathrm e}^{\tilde f(t)} } \leq \frac{e^{\tilde f({\rm C})}}{\int_{-\infty}^{\rm C} {\mathrm e}^{\tilde f(t)}\,\d t}=\frac{{\mathrm e}^{ f({\rm C})}}{\int_{-\infty}^{\rm C} {\mathrm e}^{ f(t)}\,\d t }=  \mu_{\ms}=h
\end{equation}
which leads to a contradiction. Therefore $f'({\rm C})=h$ and by concavity of $f$,  $f'=h$ on right-hand. 
Notice that the inequalities in \eqref{eq1d:intro} must be equalities. So \[\int_{-\infty}^{\rm C} e^{\tilde f(t)}\,\d t=\int_{-\infty}^{\rm C} e^{f({\rm C})+h(t-{\rm C})}\,\d t\]   which proves the proposition.

\end{proof}

%%%%%%%%%%%%%%%%%%%%%%%%%%%%
\subsection{Rigidity for  RCD spaces}\label{subsec:rigidity}
In the rest of this section,  we study rigidity of the isoperimetric inequality in $\rcd$ spaces. Recall that the Sobolev space $W^{1,2}\ms$ is a Hilbert space,  as a part of the definition of the RCD condition (cf. \cite{AGS-M,  AGMR-R}).   In this case, for $u, v\in  W^{1,2}\ms$,  we define 
$$
\nabla u \cdot \nabla v: = \inf_{\epsilon > 0} \frac{ |\D (v+ \epsilon u)|^{2} - |\D v|^{2} }{2\epsilon},
$$
and we have $\nabla u \cdot \nabla v= \nabla v \cdot \nabla u$.  Here  $|\D u|$ denotes the weak  upper gradient  of  $u$ satisfying
\begin{eqnarray*}
 \int |\D u|^2\,\d \mm=\inf \left\lbrace \liminf_{n \to \infty} \int \lip{u_n}^2\d\mm : u_n \in \Lip_{c}(X, \d),\ \! u_n \to u \text{ in } L^2 \right\rbrace 
\end{eqnarray*}
where 
\begin{equation*}
\lip f(x):=\limsup_{y\to x}  \frac{|f(y)-f(x)|}{\d(x, y)} \; \text{ if $x$ is not isolated}, \quad \lip f(x)=0 \; \text{ otherwise}.
\end{equation*}

\begin{definition}[Measure valued Laplacian, cf. \cite{G-O}]\label{D:Laplace}
Let $\Omega\subset X$ be an open subset and let $u \in W^{1,2}_{loc}\ms $. We say that $u$ is in the domain of the Laplacian, and write $u \in \D({\bf \Delta},\Omega)$, provided  there exists a signed measure $\mu$ on $\Omega$ 
such that for any $f \in \Lip_{c}(\Omega)$ it holds  that
\begin{equation}\label{eq:defTfLap}
 \int \nabla f \cdot \nabla u \,\d\mm = - \int f\,\d \mu.
\end{equation}
If  $\mu$ is unique,  we denote it by ${\bf \Delta} u$. If ${\bf \Delta} u \ll \mm$,   we write $u\in {\rm D}({\Delta}, \Omega)$ and denote its density by $\Delta u$.
\end{definition}

\medskip

In the next theorem we will prove the rigidity of the isoperimetric inequality. The strategy for proving the theorem is as follows.  Supposing there is  $\Omega \subset X$ with positive measure attaining the equality, we first show the existence of a curve $(\mu_\sigma)_{\sigma>0}$ in the Wasserstein space, such that $\supp \mu_\sigma \subset \Omega^\sigma$ and $\sigma \mapsto \ent{\mm}(\mu_\sigma)$ is linear. Then by showing that the Kantorovich potential $\phi_\sigma$ associated with $\mu_0$ and $\mu_\sigma$ is an affine function,  we  prove  $X$   splits off a one-dimensional space.
This can be seen as a dimension-free version of De Philippis--Gigli's theorem \cite{DPG-F} in the ${\rm RCD}(0, N)$ setting.   At last,  by showing the ``normal vector" of $\Omega$ is parallel to $\nabla \phi_\sigma$ in the non-smooth sense,  we prove that $\Omega$ is a half space.

\begin{theorem}[Rigidity theorem]\label{th:rigid}
Let $\ms$ be an ${\rm RCD}(0, \infty)$  metric measure space  with positive volume entropy $h_{\ms}$.

 If there is a measurable set $\Omega \subset X$ with positive measure such that 
\begin{equation}\label{th2:ass}
{\mm^+(\Omega) } =h_{\ms}  {\mm(\Omega)},
\end{equation}
then 
 \[\ms \cong \Big (\R, | \cdot |, e^{h_{\ms} t} \d t \Big) \times (Y, \d_Y, \mm_Y)\] for some  ${\rm RCD}(0, \infty)$ space $(Y, \d_Y, \mm_Y) $ with $\mm_Y(Y)<+\infty$, where the product space on the right-hand side is a metric measure space  with the canonical $L^2$-product metric and the product measure.
In  a suitable choice of coordinates,  up to a negligible set, $\Omega$ can be identified  as
\[
\Omega={(-\infty, c) \times Y} \subset  \R \times Y
\]
with  ${\mm_Y(Y)}\int_{-\infty}^c e^{h_{\ms} t} \,\d t={\mm(\Omega)}$.
\end{theorem}

\begin{proof}
The proof is divided  into five steps.

{\bf Step 1.}  We can assume that $\Omega$ is open:

By Bakry\textendash \'Emery's gradient estimate for the heat flow $f_t:=H_t (\chi_\Omega), t>0$, we can see (cf. \cite[Remark 3.5]{GH-I})
\[
\int |\D f_t|\,\d \mm\leq {\rm Per}(\Omega)\leq \mm^+(\Omega)=h_{\ms} \int f_t\,\d \mm.
\]
By Cavalieri's formula (cf. \cite[Chapter 6]{AT-T}) and the inequality \eqref{eq:per} in Theorem \ref{th1} 
\[
h_{\ms} \int  f_t\,\d \mm=h_{\ms} \int_0^1 \mm\big(\{f_t> s\}\big)\,\d s\leq  \int_0^1 {\rm Per}\big(\{f_t> s\}\big)\,\d s.
\]
By the coarea formula of Fleming\textendash Rishel (see \cite{MirandaBV} and \cite[\S 4]{ADMG-P})
\[
\int |\D f_t|\,\d \mm = \int_0^1 {\rm Per}\big(\{f_t> s\}\big)\,\d s.
\]

Combining the inequalities above, for $\mathcal L^1$-a.e. $s\in [0,1]$, we have
\[
h_{\ms}\mm\big(\{f_t>s\}\big)= \mm^+\big(\{f_t>s\}\big)
\]
By regularization of the heat flow, $f_t$ is Lipschitz (cf. \cite[THEOREM 6.5]{AGS-M}),  so $\{f_t>s\}$ is a non-trivial open set  for some $s$.
Since the isoperimetric  profile is linear, without loss of generality,  we may assume that   $\Omega$ is a connected open set.

\medskip

%%%%%%%%%%%%%%%%%%%%%%%%%%%%%%%%%%%%%%%
%第一步

\medskip

{\bf Step 2.} The potential function $\phi_\sigma$ and the optimal transport map  $T_{\phi_\sigma}$:

For $r>0$, let $\mu_{r, 0}:=\frac 1{\mm(\Omega\cap B_r)} \mm \restr{\Omega\cap B_r}$ and $\mu_{r,1}:=\frac 1{\mm\big((\Omega\cap B_r)^R\big)} \mm \restr{(\Omega\cap B_r)^R}$. Consider the  Wasserstein geodesic $(\mu_{r,t})_{t\in[0,1]}$ from $\mu_{r,0}$ to $\mu_{r,1}$.  For $\sigma>0$, by the super-position theorem  (cf. \cite[Theorem 2.10]{AG-U}  and Remark \ref{spp}),   we can see that $\mu_{r, \sigma/R}$ is concentrated on $ \Omega^{\sigma(r+R)/R}$.

By the ${\rm CD}(0, \infty)$ condition we have
\begin{equation}\label{3.22}
\ent \mm (\mu_{r, \sigma/R})\leq  -\left(1-\frac \sigma R\right)\ln \big(\mm(\Omega\cap B_r) \big) -\frac \sigma R\ln \big(\mm((\Omega\cap B_r)^{R}) \big).
\end{equation}

By \eqref{3.22} we know the family of measures $\{\mu_{r, \sigma/R}\}_{r, R>0}$ is tight (cf. \cite[Lemma 4.4]{AGMR-R}),   so it converges narrowly,   up to taking a subsequence,   to a measure $\mu_\sigma$ as $R\to \infty$ and $r\to \infty$.  By the positive volume entropy condition and  the super-position theorem again,     we can see that  $(\mu_\sigma)_{\sigma>0}$ is a Wasserstein geodesic ray and $\mu_\sigma$ is concentrated on $\overline{\Omega^\sigma}$.  From the construction,  we can see that  $\mu_{0}=\frac 1{\mm(\Omega)} \mm \restr{\Omega}$ and we may assume that  $\mm(\partial \Omega^\sigma)=0$ for all $\sigma$ without loss of generality.  By Jensen's inequality, we can also see
\begin{equation}\label{3.20}
\ent \mm (\mu_{\sigma})\geq -\ln \big(\mm(\Omega^{\sigma})\big).
\end{equation}

Letting $R\to \infty, r\to \infty$ in \eqref{3.22},  and  combining  \eqref{3.20} and the lower semi-continuity of the entropy,  we get
\begin{equation}\label{eq:step2}
-\ln \big(\mm(\Omega^{\sigma})\big)\leq\ent \mm (\mu_{\sigma})\leq \underbrace{-\ln \big(\mm(\Omega) \big) }_{=\ent \mm (\mu_{0})}-\sigma h_{\ms}.
\end{equation}
By \eqref{eq:step2}  and the assumption \eqref{th2:ass}, we can see that 
\[
\lmti{\sigma}{0}\frac{\ent \mm (\mu_{\sigma})-\ent \mm (\mu_{0})}{\sigma}\geq \lmti{\sigma}{0}\frac{-\ln \big(\mm(\Omega^{\sigma})\big)+\ln \big(\mm(\Omega)\big)}{\sigma}=-\frac{\mm^+(\Omega) }{\mm(\Omega)} =-h_{\ms}.
\]
By the ${\rm CD}(0, \infty)$ condition,  $\sigma \mapsto \ent \mm (\mu_{\sigma})$ is convex.   Combining with \eqref{eq:step2} we get
\[
-h_{\ms}\leq \frac{\d^+}{\d \sigma} \restr{\sigma=0}\ent \mm (\mu_{\sigma})\leq \frac {\ent \mm (\mu_{\sigma})-\ent \mm (\mu_{0})}{\sigma}\leq -h_{\ms}.
\]
Thus $\sigma \mapsto \ent \mm (\mu_{\sigma})$ is linear and $\ent \mm (\mu_{\sigma})=-\ln \big(\mm(\Omega) \big) - h_{\ms} \sigma$.

By \cite{RS-N},   there exist a Kantorovich potential $\phi_\sigma$, and a map $T_{\phi_\sigma}: \Omega \to \Omega^{\sigma}$,   so that $\mu_\sigma=(T_{\phi_\sigma})_\sharp \mu_0$.
 By construction of $\mu_\sigma$,  we have $W_2(\mu_0, \mu_\sigma)=\sigma$ and $\d( T_{\phi_\sigma}(x), x)=\sigma$ for $\mm$-a.e. $x\in \Omega$. Thus by the metric Brenier's theorem \cite[PROPOSITION 3.5]{AGS-M} we have $|\D \phi_\sigma|=\sigma$. Furthermore,  by  an approximation argument with simple functions, we can also prove that 
${\rm Ent}_\mm((T_{\phi_\sigma})_\sharp\nu)-{\rm Ent}_\mm(\nu)=-  h_{\ms}  \sigma$ for any $\nu$ which is concentrated on $\Omega$.

\medskip

{\bf Step 3.}  $\phi_\sigma \in {\rm D}({\bf \Delta}, \Omega)$ and ${\Delta \phi_\sigma}=\sigma h_{\ms}$.

Let $\rho$ be a Lipschitz probability density with compact support in $\Omega$.  For any $\epsilon>0$,  set $\rho_\epsilon:=c_\epsilon (\rho+\epsilon)\chi_{\Omega^\sigma}$ where $c_\epsilon$ is the normalizing constant. Let $\tau=(T_{\phi_\sigma})_\sharp (\rho \,\mm)$ and $\tau_\epsilon=(T_{\phi_\sigma})_\sharp (\rho_\epsilon \,\mm)$.  By the derivative of the entropy formula \cite[THEOREM 4.8-(b)]{AGS-M},  we have
\[
{\rm Ent}_\mm(\tau_\epsilon)-{\rm Ent}_\mm(\rho_\epsilon\,\mm) \geq -\int_{\Omega} \nabla \phi_\sigma\cdot \nabla \rho_\epsilon\,\d \mm.
\]
Letting $\epsilon \downarrow 0$,  by the monotone convergence theorem and the locality of the weak upper gradient,  we get
\[
{\rm Ent}_\mm(\tau)-{\rm Ent}_\mm(\rho\,\mm) \geq -\int_{\Omega} \nabla \phi_\sigma\cdot \nabla \rho\,\d \mm.
\]
Combining Step 2,  we get
\[
-h_{\ms} \sigma\geq -\int_{\Omega} \nabla \phi_\sigma\cdot \nabla\rho\,\d \mm.
\]
By Step 2 we know   almost all points in $\Omega$ have a pre-image of $T_{\phi_\sigma}$.  So $\{r \in \R: \mm(\Omega^{r})>0\}=\R$ where  $\Omega^{r}:=\{x\in \Omega: \d(x, \Omega^c)>|r|\}$. Considering the optimal transport induced by $- \phi_\sigma$ (cf. \cite[Proposition 5.3]{GH-C}), we can  also prove
\[
h_{\ms} \sigma\geq \int_{\Omega} \nabla \phi_\sigma\cdot  \nabla \rho\,\d \mm.
\]
Then by the Riesz\textendash Markov\textendash Kakutani representation theorem we know $\phi_\sigma \in {\rm D}({\bf \Delta}, \Omega)$ and
\[
\Delta \phi_\sigma=-h_{\ms}\sigma~~\text{on}~~\Omega.
\]

\medskip

{\bf Step 4.} The gradient flow of $\phi_\sigma$ induces an isometric  splitting.

The existence of the isometric splitting map has been well-studied  by Gigli and his co-authors  in   \cite{G-S,  GKKO-R} in  the framework of non-smooth metric measure spaces.  For convenience, we will omit some details here.

Let $\phi=\phi_\sigma /\sigma$.  From Step 2 and Step 3,   we know that $|\nabla \phi|=1$ and
$\int  |\nabla \phi|^2 \Delta \varphi\,\d \mm=\int  \Delta \varphi\,\d \mm=0$ for any $\varphi \in \Lip_c(\Omega) \cap {\rm D}(\Delta, \Omega)$.  By Step 3, $\Delta \phi=-h_{\ms}$,  so $\nabla \phi \cdot \nabla \Delta \phi=0$,  by Bochner's formula \cite[Theorem 3.3.8]{G-N} we know  $\phi$ is an affine function (in the sense of \cite[Proposition 3.2]{GKKO-R}, $\D^{\text{sym}} (\nabla \phi)=0$ and $|\D \phi|$ is constant). By   \cite[Theorem 4.4]{GKKO-R},  there is a map $F:(-\infty, 0) \times \Omega \to \Omega$,  called the {Regular Lagrangian Flow},  and studied by Ambrosio--Trevisan  \cite{AT-W} in the metric measure setting,    such that 
\begin{itemize}
\item [(i)] $F_t(\cdot):=F(t, \cdot)$ is an isometry on $X$ for each $t\in (-\infty, 0)$;
\item [(ii)] $(F(t, x))_{t\in (-\infty, 0)}$ is a geodesic (ray)  in $X$ for every $x\in \Omega$.
\end{itemize}

 By  disintegration, $\mm\llcorner_\Omega$ has a decomposition 
\begin{equation}\label{eq1:step5}
\mm\llcorner_\Omega=\int_{Y} \mm_{y}\,\d \qq (y),~~~\mm_y\in  {\rm Meas}(X_y),~~~X_y=\big\{F_t(y): t\in (-\infty,0)\big\}.
\end{equation}
Following  Cavalletti\textendash Mondino  \cite[4b]{CM-Laplacian},  with the help of \eqref{eq1:step5}, we can represent the  measure-valued Laplacian  in the following way
\[
{\bf \Delta} \phi=\int_{Y} h_{\ms}  \phi \,\d \mm_y\d \qq (y)
\]
By integration by parts  on $\R$ (cf. \cite[Theorem 4.8]{CM-Laplacian}), this implies that $\mm_y=e^{V_y}\,\d t$ with $ V'_{y}=h_{\ms}$ on $X_y$.  Furthermore, using the same argument as in Step 2, we can prove that 
\[
\ln (\mm(\Omega))\geq \ln (\mm(\Omega^{\sigma}))-h_{\ms} \sigma.
\]
Combining this inequality with  \eqref{eq:step2} we get $\mm^+(\Omega^\sigma)=h_{\ms}\mm(\Omega^\sigma)$. Since $\sigma$ is arbitrary, we can see that  $F$ can be defined on the whole product space $\R \times X$. Similar to  \eqref{eq1:step5} we can write $ \mm=\int_{Y} \mm_{y}\,\d \qq (y)$  with $\mm_y=e^{V_y}\,\d t$ and $ V'_{y}=h_{\ms}$.

 Following  \cite[Section 6]{G-S} and \cite[Section 5]{GKKO-R}, we can prove that  $F$ induces  an isometry between  $(X, \d)$ and  the product  space $  (Y, \d_Y) \times \big (\R, | \cdot |\big)$   equipped with  the $L^2$-product distance, where $Y$ can be identified as $\phi^{-1}(0)$.
Precisely, there are isometries $\Phi, \Psi$ defined by
\[
\Phi: X \ni x \mapsto (y, t)\in Y \times \R~~\text{s.t.}~F_t(y)=x
\]
and
\[
\Psi: Y \times \R \ni (y, t) \mapsto x= F_t(y) \in X.
\]
Furthermore, note that
 \[
 (F_t)_\sharp \mm=e^{ h_{\ms} t}\mm~~~~~\forall t\in \R.
 \]
We can define
 \[
 \mm_Y(A):=\lmt{\epsilon}{0} \frac{\mm\big(\Psi(A \times [0, \epsilon])\big )}{\epsilon}=\qq(A),~~~~\forall A\subset Y~~\text{is measurable},
 \]
so that
 \[
\Phi_\sharp \mm= \mm_Y \times e^{ h_{\ms} t} \d t.
 \] Using the same argument as in \cite[Section 6]{G-S} and \cite[Section 5]{GKKO-R}, we can prove that $(Y, \d_Y, \mm_Y) $ is $\rcd$ and 
$$\ms \mathop{ \cong}_{\Phi, \Psi}  (Y, \d_Y, \mm_Y) \times \Big (\R, | \cdot |, e^{h_{\ms} t}\,\d t\Big). $$

\medskip

{\bf Step 5.} Characterization of $\Omega$.

By  the decomposition  \eqref{eq1:step5} and Theorem \ref{th1},  it holds that
\[
\mm^+(\Omega) \overset{\rm Fatou} \geq  \int_{Y} \mm_{y}^+(\Omega)  \, \d \qq(y) \geq h_{\ms} \int_{Y} \mm_{y}(\Omega)  \, \d \qq(y)= h_{\ms} \mm(\Omega).
\]
Thus  
\[
\mm^+_{y}(\Omega)  =h_{\ms} \mm_{y}(\Omega) ~~~\qq\text{-a.e.}~~y\in Y.
\]
By 1-dimensional rigidity  in Proposition \ref{lemma:rigidity},  for almost every $y\in Y$, $\Omega \cap X_{y}$ is a half line, and we denote it by $(-\infty, \mathsf{b}(y)]$. So  we can identify $\Omega$ as
\[
\Omega\cong \Big \{(y, r) :  r\in \big (-\infty, \mathsf{b}(y)\big ),  ~ y\in Y,~ \mathsf{b}(y)\in \R \Big \}
\]
and  $\partial \Omega$ is the graph of a measurable  function $\mathsf{b}(\cdot)$ on $Y$.

\paragraph{Claim:} $\mathsf{b}$ is  a constant function. 
In the smooth case, the optimality of $\Omega$ surely implies  that the line $t \mapsto F_t(x)$ is ``vertical" to the boundary of $\Omega$.
If  $\mathsf{b}$ is  not constant,  the ``normal vector" of $\partial \Omega$ has a non-trivial ``horizontal component" which leads to a contradiction. See Remark \ref{rk} for more explanations and related references.

From the proof of Theorem \ref{th1}, we know there is a sequence of Lipschitz functions $(f_n)_{n\in \N}$ such that $f_n \to \chi_\Omega$ in $L^1$ and $${\rm Per}(\Omega)=\mm^+(\Omega)=\lmt{n}{+\infty} \int |\D f_n|\,\d \mm.$$ 
For simplicity, we write $f_n=f_n(y, r)$ as a function on $Y \times \R$, and $\mm=\mm_Y \times \mm_\R$ where $ \mm_\R= e^{ h_{\ms} t} \d t$. Set $f^r_n=f_n(\cdot, r)$, $f^y_n=f_n(y, \cdot)$ and $\chi^y_\Omega=\chi_{\Omega^y}=\chi_{\Omega\cap \{(y, r):r\in \R\}}$. By Fubini's theorem, $f_n \to \chi_\Omega$ in $L^1$ implies that
\[
\int_Y \left(\int_\R |f^y_n(t)-\chi^y_\Omega| \,\d \mm_\R\right) \d\mm_Y \to 0~~~\text{as}~n\to \infty.
\]
So there is a subsequence of $(f_n)$, still denoted by $(f_n)$, such that 
\[
\lmt{n}{\infty}\int_\R |f^y_n(t)-\chi^y_\Omega| \,\d \mm_\R=0,~~~\mm_Y\text{-a.e.}~y\in Y,
\]
and
\begin{equation}\label{eq0:step7}
\lmt{n}{\infty}\int_\R  f^y_n(t) \,\d \mm_\R=\int_\R \chi^y_\Omega \,\d \mm_\R=\frac 1{h_{\ms}} e^{ h_{\ms}\mathsf{b}(y)},~~~\mm_Y\text{-a.e.}~y\in Y.
\end{equation}
So by lower semi-continuity,
\begin{equation}\label{eq1:step7}
\lmti{n}{\infty} \int_\R |\D f^y_n(t)| \,\d \mm_\R \geq \mm_\R^+(\Omega^y)=e^{h_{\ms} \mathsf{b}(y)}.
\end{equation}

 By \cite[Theorem 5.2]{AGS-B}, $|\D f_n|^2=|\D f^r_n|^2+|\D f^y_n|^2$,  where $|\D f^r_n|=|\D f^r_n|_Y$ is the weak gradient of $f^r_n$ in $Y$, and $|\D f^y_n|=|\D f^y_n|_\R$ is the weak gradient of $f^y_n$ in $\R$ which is the norm of partial derivatives. So for any $\epsilon>0$ we have
\begin{eqnarray*}
&& \int |\D f_n|\,\d \mm = \int  \sqrt{|\D f^r_n|^2+|\D f^y_n|^2}\,\d \mm_\R \d\mm_Y \\
&=&  \int_{\{|\D f^r_n|>\epsilon |\D f^y_n|\}}  \left(\frac{|\D f^r_n|^2}{\sqrt{|\D f^r_n|^2+|\D f^y_n|^2}+|\D f^y_n|}+|\D f^y_n|\right)\,\d \mm_\R \d\mm_Y \\ 
&&+ \int_{|\D f^r_n|\leq \epsilon |\D f^y_n|}  \sqrt{|\D f^r_n|^2+|\D f^y_n|^2}\,\d \mm_\R \d\mm_Y\\
&\geq &\int_{\{|\D f^r_n|>\epsilon |\D f^y_n|\}}  \left(\frac{|\D f^r_n|}{2\sqrt{1+\epsilon^{-2}}}+|\D f^y_n|\right)\,\d \mm_\R \d\mm_Y
+ \int_{|\D f^r_n|\leq \epsilon |\D f^y_n|}  |\D f^y_n|\,\d \mm_\R \d\mm_Y.
\end{eqnarray*}
Then
\begin{eqnarray*}
 \int |\D f_n|\,\d \mm 
&\geq &\int_{\{|\D f^r_n|> \epsilon |\D f^y_n|\}}  \frac{|\D f^r_n|}{2\sqrt{1+\epsilon^{-2}}}\,\d \mm
+ \int |\D f^y_n|\,\d \mm\\
&\geq &\int  \frac{|\D f^r_n|}{2\sqrt{1+\epsilon^{-2}}}\,\d \mm
+\left (1-\frac \epsilon {2\sqrt{1+\epsilon^{-2}}}\right ) \int |\D f^y_n|\,\d \mm.
\end{eqnarray*}
Letting $n\to \infty$ and combining the inequalities above with \eqref{eq1:step7}, we get
\begin{eqnarray*}
&&\mm^+(\Omega)= \lmt{n}{\infty} \int |\D f_n|\,\d \mm \\
&\geq &\lmti{n}{\infty}  \int  \frac{|\D f^r_n|}{2\sqrt{1+\epsilon^{-2}}}\,\d \mm
+(1-\frac \epsilon {2\sqrt{1+\epsilon^{-2}}}) \int \mm_\R^+(\Omega^y)\, \d \mm_Y\\
&\geq &  \lmti{n}{\infty}  \int  \frac{|\D f^r_n|}{2\sqrt{1+\epsilon^{-2}}}\,\d \mm
+(1-\frac \epsilon{2\sqrt{1+\epsilon^{-2}}}) h_{\ms} \int \mm_\R(\Omega^y)\, \d \mm_Y\\
&=&   \lmti{n}{\infty}  \int  \frac{|\D f^r_n|}{2\sqrt{1+\epsilon^{-2}}}\,\d \mm
+(1-\frac \epsilon{2\sqrt{1+\epsilon^{-2}}}) h_{\ms} \mm(\Omega).
\end{eqnarray*}
Combining this with $\mm^+(\Omega)=h_{\ms} \mm(\Omega)$ we get
\[
\epsilon \mm^+(\Omega)  \geq  \lmti{n}{\infty}  \int  {|\D f^r_n|}\,\d \mm.
\]
Letting $\epsilon \to 0$ we obtain
\begin{equation}\label{eq2:step7}
\lmti{n}{\infty}  \int  {|\D f^r_n|}\,\d \mm=0.
\end{equation}

Define   $g_n(y)=\int_\R  f_n(y, r) \,\d \mm_\R(r)$.
We can approximate $g_n$  in $L^1(\mm)$ with  functions of the form  $\sum_{k\in I, |I|<\infty}  c_k f_n^{r_k}(y)$,  and   approximate  $ \int_\R  |\D f^r_n| \,\d \mm_\R(r)$  with  functions of the form  $\sum_{k\in I, |I|<\infty}  c_k |\D f_n^{r_k}|(y)$.  Then by a diagonal argument we can approximate $g_n$  in $L^1(\mm)$ with Lipschitz functions of the form  $\sum_{k\in I, |I|<\infty}  c_k h_k(y)$,  and  approximate  $ \int_\R  |\D f^r_n| \,\d \mm_\R(r)$ with $\sum_{k\in I, |I|<\infty}  c_k |\D h_k|$.  Combining this approximation with the lower semi-continuity, 
one can prove  $$|\D g_n| \leq  \int_\R  |\D f^r_n| \,\d \mm_\R(r).$$ 
Thus \eqref{eq2:step7} implies
\[
\lmti{n}{\infty}  \int |\D g_n|\,\d \mm_Y \leq  \lmti{n}{\infty}  \int  |\D f^r_n|\,\d \mm = 0.
\]
By  \eqref{eq0:step7}  and the lower semi-continuity again, we know $\int |\D  e^{ h_{\ms}\mathsf{b}(y)}|\,\d \mm_Y(y)=0$ and $\mathsf{b}(\cdot)$ is constant.

\end{proof}

\bigskip
\brk\label{rk}
From the proof of Theorem  \ref{th1}, we can see that $\chi_\Omega$ is a BV function. By a general  integration-by-parts formula of Brena--Gigli \cite[Theorem 4.13]{BrenaGigli24}, we have
\[
\int_\Omega {\rm div} v\,\d \mm=-\int v\cdot v_\Omega\,\d \mu 
\]
for any ``sufficiently smooth" vector field  $v$,   where $\mu$ is the total variation of $\chi_\Omega$ and $v_\Omega$ plays the role of the ``outward normal vector" of $\Omega$.

Suppose that $\Omega$ attains the equality in the sharp isoperimetric inequality. If
 $X$ is a Riemannian manifold equipped with a product metric structure induced by $\nabla \phi$,  from the proof of  Theorem \ref{th:rigid} and  integration by parts, we can see that the  normal vector  $v_\Omega=\nabla \phi$, so $\Omega$ is surely a half space (i.e., the function $e$ in the proof of previous theorem is constant).  In the setting of $\rcdkn$ spaces, Antonelli--Brena--Pasqualetto \cite[Theorem 3]{ABP24} identified the sub-graph of  BV functions  on  $Y$ and  sets of finite perimeter on the cylinder $Y\times \R$.
Using their result,  one can also see that $\Omega$ is a half space and the proof can be {\bf significantly shortened}.  Unfortunately, it is still unknown whether \cite[Theorem 3]{ABP24} is  valid for general $\rcd$ spaces or not.

\erk
%%%%%%%%%%%%%%%%%%%%%%%%%%%
\def\cprime{$'$}

\end{document}